\documentclass[11pt]{article}
\date{}

\usepackage[title]{appendix}
\usepackage{xcolor}
\usepackage[margin=1in]{geometry}                
\usepackage{graphicx}
\usepackage{subcaption}
\usepackage{pdflscape}
\usepackage{amssymb}
\usepackage[normalem]{ulem}
\usepackage{hyperref}
\usepackage{enumitem}
\usepackage{mathrsfs}
\usepackage{epstopdf}
\usepackage{rotating}
\usepackage{longtable} 
\usepackage{adjustbox}
\usepackage{float}

\usepackage{color}
\usepackage{bbm, dsfont}
\usepackage{pst-node}
\usepackage{tikz-cd}
\usepackage{amsfonts} 
\usepackage{geometry}
\usepackage{amsthm}
\usepackage{amsmath}
\usepackage{titlesec}
\usepackage{amssymb}
\usepackage{enumitem}
\usepackage{float}
\usepackage [english]{babel}
\usepackage [autostyle, english = american]{csquotes}
\usepackage{algorithm}
\usepackage[noend]{algpseudocode} 
\makeatletter
\def\BState{\State\hskip-\ALG@thistlm}
\makeatother
\usepackage{hyperref}
\usepackage[normalem]{ulem}
\usepackage{mathrsfs}
\usepackage[italicdiff]{physics}

\newlist{casess}{enumerate}{1}
\setlist[casess]{label=     \textbf{Case} \arabic*:}
\usepackage{mathtools}

\makeatletter
\newcommand*{\rom}[1]{\expandafter\@slowromancap\romannumeral #1@}
\makeatother

\usepackage{etoolbox}

\makeatletter
\patchcmd{\ttlh@hang}{\parindent\z@}{\parindent\z@\leavevmode}{}{}
\patchcmd{\ttlh@hang}{\noindent}{}{}{}
\makeatother

\usepackage{listings}
\usepackage{color} 
\definecolor{mygreen}{RGB}{28,172,0} 
\definecolor{mylilas}{RGB}{170,55,241}

\newlist{Assumptions}{enumerate}{1}
\setlist[Assumptions]{label=     \textbf{Assumption} \arabic*:}

\makeatletter

\newsavebox{\@brx}
\newcommand{\llangle}[1][]{\savebox{\@brx}{\(\m@th{#1\langle}\)}%
  \mathopen{\copy\@brx\kern-0.5\wd\@brx\usebox{\@brx}}}
\newcommand{\rrangle}[1][]{\savebox{\@brx}{\(\m@th{#1\rangle}\)}%
  \mathclose{\copy\@brx\kern-0.5\wd\@brx\usebox{\@brx}}}
\makeatother

\usepackage{lipsum} 
\usepackage{titlesec}
\titleformat{\subsection}[runin]
       {\normalfont\bfseries}
       {\thesubsection}
       {0.5em}
       {}
       [.]

 \newtheorem{thm}{Theorem}[section]
 \newtheorem{cor}[thm]{Corollary}
 
 \newtheorem{lem}[thm]{Lemma}
 
 \theoremstyle{definition}
 \newtheorem{defn}[thm]{Definition}
 \theoremstyle{remark}

 \numberwithin{equation}{section}

\numberwithin{equation}{section}

\def\N{\mathbb{N}}

\def\Z{\mathbb{Z}}

\def\R{\mathbb{R}}
\def\Z{\mathbb Z}

\DeclarePairedDelimiterX{\inp}[2]{\langle}{\rangle}{#1, #2}

\makeatletter
\newcommand*\bigcdot{\mathpalette\bigcdot@{.5}}
\newcommand*\bigcdot@[2]{\mathbin{\vcenter{\hbox{\scalebox{#2}{$\m@th#1\bullet$}}}}}
\makeatother

\def\CC{\mathbb C}

\def\<{\langle}
\def\>{\rangle}

\numberwithin{equation}{section}

\usepackage[backend=biber,maxnames=10]{biblatex}
\begin{document}

\title{Rapid decay and functional calculus in ${\rm C}^*$-probability spaces}

\author{Felipe Flores
}

\maketitle

\begin{abstract}\setlength{\parindent}{0pt}\setlength{\parskip}{1ex}\noindent
 Let $(A,\rho)$ be a tracial ${\rm C}^*$-probability space with the rapid decay property relative to a filtration $L$. We show that the associated Sobolev algebra $H_L^{\infty}(A,\rho)$ is closed under the smooth functional calculus of $A$. Some consequences include norm estimates, the ideal separation property, and the fact that the inclusion $H_L^{\infty}(A,\rho)\subset A$ induces an isomorphism in $K$-theory.
\end{abstract}


\section{Introduction}

The rapid decay property, as introduced by Jolissaint \cite{Jo89,Jo90}, is a property of length functions on countable groups that provides a polynomial-based estimate for the operator norm of convolvers acting on the $\ell^2$-space of the group. More concretely, a countable group $G$ has the rapid decay property with respect to the length function $l:G\to[0,\infty)$ if there exists a polynomial $P$ such that 
$$
\|f*h\|_2\leq P(r)\|f\|_2\|h\|_2
$$
holds for every $h\in \CC G$, and every $f\in\CC G$ that is supported in the ball of radius $r$ with respect to $l$. 

Large classes of interesting groups have the rapid decay property with respect to a natural length function. For example, groups of polynomial growth, hyperbolic groups, Coxeter groups, mapping class groups, and many more (see \cite{Ch17}).

Such a property has classical applications to approximation properties for group ${\rm C}^*$-algebras \cite{Ha78}, the study of the Baum-Connes conjecture \cite{La02}, and it is closely related to Khintchine-type inequalities \cite{RiXu06}. Recently, it has been used to compute the stable rank in some group ${\rm C}^*$-algebras \cite{DydlH99,GeOs20}, to produce examples of selfless ${\rm C}^*$-algebras \cite{AGKEP25,KEPT25,HKER25}, and simple ${\rm C}^*$-algebras associated with quantum groups \cite{VaVe07} and $q$-Gaussian variables \cite{AmJeWa26}. Moreover, the rapid decay property is also featured in many of the recent developments in random matrix theory (\cite[Section 7.1]{bordenave2024normmatrixvaluedpolynomialsrandom}, \cite[Appendix B]{NewApproachII}, \cite[Sections 5 and 6]{MdLS},\cite[Section 2.6]{vanhandel2026strongconvergencephenomenon}).

One of the most interesting features of the rapid decay property is that it allows us to properly define a subalgebra of 'rapidly decaying functions,' typically denoted by $H^\infty_L(G)$. Indeed, the idea is as follows: Consider the ${\rm C}^*$-algebra $C(\mathbb T^k)$ of continuous functions on the $k$-torus. Note, however, that the entirety of $C(\mathbb T^k)$ becomes impractical for many purposes, such as encoding the geometric properties of $\mathbb T^k$, or defining derivations. A much more educated choice is to consider the Fr\'echet subalgebra of infinitely differentiable functions $C^\infty(\mathbb T^k)$. However, note that the Fourier-Gelfand transform isometrically maps $C(\mathbb T^k)$ into the group ${\rm C}^*$-algebra $C^*(\Z^k)$, while $C^\infty(\mathbb T^k)$ is sent to the space of sequences that decay rapidly, which is given by
$$
H^\infty(\Z^k)=\{a:\Z^k\to\CC: \sum_{n\in\Z^k} |a(n)|^{2}(1+|n|)^{2s}<+\infty, \text{ for all }s\geq0 \}.
$$
This identification extends very naturally to a definition in the setting of noncommutative groups equipped with length functions. Indeed, in \cite[Definition 1.1.6]{Jo90}, Jolissaint introduced the Sobolev algebra $H^\infty_L(G)$, explicitly given by 
$$
H^\infty_L(G)=\{a:G\to\CC: \sum_{g\in G} |a(g)|^{2}(1+L(g))^{2s}<+\infty, \text{ for all }s\geq0 \}.
$$
Furthermore, the space $H^\infty_L(G)$ is contained in $C^*_r(G)$ precisely when $G$ has rapid decay with respect to the length function $L$ \cite[Definition 1.2.1, Proposition 1.2.6]{Jo90}, so one could argue that the study of Sobolev subalgebras is equivalent to the study of rapid decay itself.

Our purpose in this article is to study the spectral properties of the Sobolev algebra associated with a ${\rm C}^*$-probability space with the rapid decay property, as introduced in \cite{HKER25}. More precisely, the idea is to prove that such an algebra is closed under the smooth functional calculus of the ambient ${\rm C}^*$-algebra and to derive some interesting consequences of this fact. In particular, we wish to provide norm estimates applicable to random matrix theory and prove that the inclusion $H_L^{\infty}(A,\rho)\subset A$ induces an isomorphism in $K$-theory. In the case of countable groups, this last observation was one of the crucial ingredients that allowed Lafforgue to famously settle the Baum-Connes conjecture for hyperbolic groups \cite{La02}. 

Other works that establish results about permanence under the smooth functional calculus are \cite{MB79,BaElJo84,BC91,KiSh94,Fl25a,Fl25b}.

Our article also fits within a larger program that has attracted the attention of many researchers over the last few years. This program is based on the idea of extending the rapid decay property, or at least adapting its methods to larger classes of algebras. Examples of such new settings include ${\rm C}^*$-algebras associated with \'etale groupoids \cite{Ho17}, Fell bundles \cite{BuKa26}, $L^p$-operator algebras \cite{LiYu17,AOP25}, and abstract ${\rm C}^*$-probability spaces \cite{HKER25}. 

Indeed, following Hayes-Kunnawalkam Elayavalli-Robert \cite{HKER25}, we will say that an abstract ${\rm C}^*$-probability space $(A,\rho)$ has the rapid decay property with respect to a filtration $L=(L_{n})_{n=0}^{\infty}$ if there are constants $C>0$ and $\alpha>0$ such that 
    \[
    \|a\|\leq C(n+1)^{\alpha}\|a\|_{2}
    \]
    for all $a\in L_{n}$. Such a definition allows for the definition of a natural Sobolev algebra, here denoted $H_L^{\infty}(A,\rho)$. As a set, it has the following form: 

    $$
    H_L^{\infty}(A,\rho)=\big\{a\in A:\sum_{k=0}^{\infty}(k+1)^{2\alpha}\|p_{k}(\xi)\|_{2}^{2}<+\infty, \text{ for all }\alpha> 0\big\}.
    $$
We refer to the preliminaries section for a fully rigorous definition.

Let us now state our main theorem.

\begin{thm}[see Theorem \ref{actual}]\label{mainthm}
        Let $(A,\rho)$ be a tracial ${\rm C}^*$-probability space, and let
$L=(L_n)_{n=0}^\infty$ be a filtration of $A$ with the rapid decay property.
Then $H_L^\infty(A,\rho)$ is closed under the smooth functional calculus of
$A$.

More precisely, fix rapid decay constants $C,\alpha>0$.  For every
$n\in\mathbb N$, there is a constant $K_n>0$ such that, whenever
$a=a^*\in H_L^\infty(A,\rho)$ and $f\in C_c^\infty(\mathbb R)$,
\begin{equation*}
 f(a)=\frac{1}{2\pi}\int_{\mathbb R}\widehat f(t)e^{ita}\,dt
\end{equation*}
converges in every subspace $H_L^n(A,\rho)$, and we have the norm estimate
\begin{equation*}
 \|f(a)\|_{H_L^n}\leq K_n\bigl(1+\|a\|_{H_L^{\alpha+n+1}}\bigr)^n \int_{\mathbb R}|\widehat f(t)|(1+|t|)^n\,dt.
\end{equation*}
\end{thm}

An immediate consequence of such a robust calculus is the fact that $H_L^{\infty}(A,\rho)$ can detect and `reconstruct' the ideals of $A$ (see \cite{KiSh94}). Another immediate consequence is our promised result stating that $H_L^{\infty}(A,\rho)$ and $A$ have the same $K$-theory groups (see \cite{Ph91,Sc92}). We record these results in the following corollary.

\begin{cor}
    Let $(A,\rho)$ be a tracial ${\rm C}^*$-probability space with the rapid decay property relative to a filtration $L$. Then the following are true:
    \begin{enumerate}
        \item[(i)] Let $I\unlhd A$ be a closed two-sided ideal. Then $\overline{I\cap H_L^{\infty}(A,\rho) }^{\|\cdot\|}=I$.  
        \item[(ii)] The inclusion $H_L^{\infty}(A,\rho)\subset A$ induces an isomorphism in $K$-theory.
    \end{enumerate}
\end{cor}

We now briefly comment on the proof of Theorem \ref{mainthm}. The idea is to define a densely-defined, unbounded $^*$-derivation $\partial:{\rm dom}(\partial)\subset \mathcal B_L\to \mathcal B_L$ and use it to identify the Fr\'echet subalgebra $H_L^{\infty}(A,\rho)$ with the set of vectors in $A$ that lie in the domain of $\partial^n$ for all $n\in\N$. Such a strategy is motivated by the work of Ji \cite{Ji92}, who used this approach in the group setting to demonstrate that, under the rapid decay property, Sobolev subalgebras are closed under the holomorphic functional calculus of the reduced group ${\rm C}^*$-algebra. Similar results were obtained in \cite{LiYu17,AOP25,ChWa26}.

After identifying $H_L^{\infty}(A,\rho)$, we follow the work of Bratteli-Elliott-Jorgensen, who proved that, for any unbounded closed $*$-derivation $\partial$ on $A$, ${\rm dom}(\partial^n)$ is closed under the smooth functional calculus of $A$ and provided an integral formula for said calculus. We iterate their formula and use it to obtain a polynomial bound on the growth of all self-adjoint elements in ${\rm dom}(\partial^n)$. This naturally leads to a smooth functional calculus based on the Fourier inversion formula, as our main theorem demonstrates. This type of functional calculus is known as the Dixmier-Baillet construction \cite{Di60,MB79}. Other functional calculi \`a la Dixmier-Baillet have been constructed in \cite{KiSh94,Fl25a,Fl25b}.

\section{Preliminaries}

For a faithful $C^{*}$-probability space $(A,\rho)$, we denote the GNS Hilbert space by $L^2(A,\rho)$. Recall that every element $a\in A$ can be naturally viewed as a vector $\hat a\in L^2(A,\rho)$, and that the action of $A$ by left multiplication on itself extends to all of $L^2(A,\rho)$. 

\begin{defn}\label{defn:RDP}
Let $(A,\rho)$ be a unital $C^{*}$-probability space. A \emph{filtration} of $A$ is an increasing sequence $(L_{n})_{n=0}^{\infty}$ of linear subspaces of $A$ such that:
\begin{itemize}
    \item $L_{0}=\mathbb C 1$;
    \item $L_{n}$ is stable under the adjoint operation $*$;
    \item $L_{n}L_{k}\subseteq L_{n+k}$ for all $n,k$;
    \item $\overline{\bigcup_{k}L_{k}}^{\|\cdot\|}=A$.
\end{itemize}
\end{defn}

One key example of interest is the following. Suppose that $x=(x_{1},\cdots,x_{r})\in A^{r}$ and that $A=C^{*}(x,1)$. In this case, we obtain a filtration on $A$ by setting 
\[
L_{n}=\{P(x):P\in \mathbb C^{*}\langle{T_{1},\cdots,T_{r}}\rangle \textnormal{ and } \deg(P)\leq n\}
\]
for $n\in \N$. We call $(L_{n})_{n=0}^{\infty}$ the degree-filtration obtained from the generating tuple $x$.

\begin{defn}[\cite{HKER25}]
Given a ${\rm C}^*$-probability space $(A,\rho)$ with filtration $(L_{n})_{n=0}^{\infty}$, we say that $(L_{n})_{n=0}^{\infty}$ has the \emph{rapid decay property}, or that $(A,\rho)$  has rapid decay relative to $(L_{n})_{n=0}^{\infty}$, if 
there are constants $C>0$ and $\alpha>0$ such that 
    \[
    \|a\|\leq C(n+1)^{\alpha}\|\hat a\|_{2}
    \]
    for all $a\in L_{n}$.
  When $(L_{n})_{n=0}^{\infty}$ is the degree-filtration of $C^{*}(x,1)$ coming from a tuple $x\in A^r$, we say that the tuple $x$ has rapid decay.  
\end{defn}

For faithful $C^{*}$-probability spaces, we can characterize the rapid decay property via locally convex spaces as follows.

\begin{defn}
Let $(A,\rho)$ be a unital $C^{*}$-probability space, with $\rho$ faithful. By faithfulness, we may identify $A\subseteq L^{2}(A,\rho)$. 

Let $(L_{n})_{n=0}^{\infty}$ be a filtration on $A$. For $k\in \N$, let $q_{k}\in \mathbb B(L^2(A,\rho))$ be the projection onto $\overline{L_{k}}^{\|\cdot\|_2}$, and set $p_{k}=q_{k}-q_{k-1}$, with $q_{-1}=0$. For $\xi\in L^{2}(A,\rho)$ and $\alpha>0$, we define
\begin{align*}
\|\xi\|_{H^{\alpha}_L} &=\left(\sum_{k=0}^{\infty}(k+1)^{2\alpha}\|p_{k}(\xi)\|_{2}^{2}\right)^{1/2}.
\end{align*}

Define vector subspaces
\begin{align*}
H_L^{\alpha}(A,\rho) &=\{\xi\in L^{2}(A,\rho):\|\xi\|_{H^{\alpha}_L}<+\infty\},
\end{align*}
equipped with the norm $\|\cdot\|_{H^{\alpha}_L}$. It can be shown that $H_L^{\alpha}(A,\rho)$ is complete.
We let 
\begin{align*}
H_L^{\infty}(A,\rho)&=\bigcap_{\alpha>0}H_L^{\alpha}(A,\rho).
\end{align*}
\end{defn}

We give $H_L^{\infty}(A,\rho)$ a locally convex structure by equipping it with the family of norms $(\|\cdot\|_{H^{\alpha}_L})_{\alpha>0}$. Restricting these families of seminorms to $\alpha\in \N$ defines the same topology; thus, this locally convex topology is completely metrizable. 

It was established in \cite[Proposition 3.3]{HKER25} that any unital ${\rm C}^*$-probability space with the rapid decay property must be faithful. Hence, restricting our scope to faithful ${\rm C}^*$-probability spaces will produce no loss in generality.

\section{Main results}

From now on, we fix a faithful tracial ${\rm C}^*$-probability space $(A,\rho)$ and a filtration $L=(L_{n})_{n=0}^{\infty}$ of $A$. We use the filtration to define the following unbounded operator:
$$
D_L:{\rm dom}(D_L)\subset L^{2}(A,\rho)\to L^{2}(A,\rho),\quad \text{ defined by }D_L(\xi)=\sum_{k\in\N}(1+k)p_k(\xi).
$$
Here, the domain of $D_L$ is obviously given by
$$
{\rm dom}(D_L)=\{\xi\in  L^{2}(A,\rho): D_L(\xi)\in L^{2}(A,\rho)\},
$$
and it contains $\bigcup_{n\in\N} L_n$, so $D_L$ is densely defined.
\begin{lem}\label{closedd}
    Let $(A,\rho)$ be a tracial ${\rm C}^*$-probability space, and let $L=(L_{n})_{n=0}^{\infty}$ be a filtration of $A$. Then, the associated operator $D_L$ is closed and self-adjoint.
\end{lem}
\begin{proof}
    We will start by showing closedness. Let $\xi_n\in {\rm dom}(D_L)$ such that $\xi_n\to \xi$ and $D_L(\xi_n)\to\eta$. Then, for all $k\in\N$, one has $p_k(\xi_n)\to p_k(\xi)$, and applying Fatou's lemma, we see that 
    \begin{align*}
        \liminf_{n\to\infty}\sum_{k\in\N}(1+k)^2\|p_k(\xi_n)\|_2^2\geq \sum_{k\in\N}\lim_{n\to\infty}(1+k)^2\|p_k(\xi_n)\|_2^2=\sum_{k\in\N}(1+k)^2\|p_k(\xi)\|_2^2.
    \end{align*}
    However, we can apply the Parseval identity to identify the left-hand side as the limit inferior of
    $$
    \sum_{k\in\N}(1+k)^2\|p_k(\xi_n)\|^2_2=\Big\|\sum_{k\in\N}(1+k)p_k(\xi_n)\Big\|^2_2=\|D_L(\xi_n)\|^2_2.
    $$
    Since the sequence $D_L(\xi_n)\in L^{2}(A,\rho)$ is bounded, we conclude that $D_L(\xi)\in L^{2}(A,\rho)$; thus, $\xi\in {\rm dom}(D_L)$. A similar reasoning also shows that 
    \begin{align*}
        \|\eta-D_L(\xi)\|_2^2&=\sum_{k\in \N}\|p_k(\eta)-p_k(D_L(\xi))\|_2^2 \\
        &=\sum_{k\in \N}\|p_k(\eta)-(1+k)p_k(\xi)\|_2^2\\
        &=\sum_{k\in \N}\lim_{n\to \infty}\|p_k(\eta)-(1+k)p_k(\xi_n)\|_2^2 \\
        &\leq\liminf_{n\to \infty} \sum_{k\in \N}\|p_k(\eta)-p_k(D_L(\xi_n))\|_2^2 \\
        &\leq \liminf_{n\to \infty} \|\eta-D_L(\xi_n)\|_2^2=0.
    \end{align*}
    Hence $\eta=D_L(\xi)$ and we conclude that $D_L$ is closed. 

    To see that $D_{L}$ is self-adjoint, first note that $D_{L}$ is symmetric. So it suffices to show that if $\xi\in {\rm dom}(D_{L}^{*})$, then $\xi\in {\rm dom}(D_{L})$. Suppose $\xi\in {\rm dom}(D_{L}^{*})$, and choose $\eta\in {\rm dom}(D_{L})$. For $k\in\N$ and using that the domain of $D_{L}$ contains the image of $p_{k}$, we see that
    $$
    \langle \eta,p_k(D_L^*\xi)\rangle=\langle D_Lp_k(\eta),\xi\rangle=(1+k)\langle \eta,p_k(\xi)\rangle,
    $$
    hence $p_k(D_L^*\xi)=(1+k)p_k(\xi)$. This means that 
    $$
    \sum_{k\in\N}(1+k)^{2}\|p_{k}(\xi)\|^{2}=\sum_{k\in\N}\|p_{k}(D_L^*\xi)\|^{2}=\|D_L^*\xi\|^{2}<+\infty.
    $$
    Thus $\xi\in {\rm dom}(D_{L})$. \end{proof}

By Stone's theorem, $D_L$ generates a strongly continuous one-parameter group $U_t=e^{itD_L}\in \mathbb B(L^2(A,\rho)) $. For $t\in \R$ and $T\in  \mathbb B(L^2(A,\rho)) $, we set 
$$
\beta_t(T)=U_tTU_t^*
$$
and 
$$
\mathcal B_L=\{T\in \mathbb B(L^2(A,\rho)):t\mapsto\beta_t(T)\text{ is norm continuous}\}.
$$
Then $\mathcal B_L$ is a unital ${\rm C}^*$-subalgebra of $\mathbb B(L^2(A,\rho))$ that is invariant under the action of $\beta$ and such that $(\beta_t)_{t\in\mathbb R}$ restricts to a strongly continuous one-parameter group of $*$-automorphisms of $\mathcal B_L$.  Let
\begin{equation*}
 \partial(T)=\lim_{t\to0}\frac{\beta_t(T)-T}{t}
\end{equation*}
denote its infinitesimal generator. Then $\partial$ is a closed, densely defined $*$-derivation of $\mathcal B_L$ in the following sense.

\begin{defn}
    Let $B$ be a ${\rm C}^*$-algebra. A \emph{closed derivation} $\partial$ of $B$ is a linear map from a subalgebra $A\subset B$ into $B$ that satisfies \begin{enumerate}
    \item[(i)] $\partial(ab)=a\partial(b)+\partial(a)b$, for all $a,b\in A$.
    \item[(ii)] If $a_n\in A$, $a_n\to a$ and $\partial(a_n)\to b$, then $a\in A$ and $b=\partial(a)$.
\end{enumerate} If, in addition, $x\in A$ implies $x^*\in A$ and $\partial(x^*)=\partial(x)^*$, then $\partial$ is called a \emph{closed $*$-derivation}. 
\end{defn}

Now, the idea is to study the powers of the derivation $\partial$. In particular, we wish to describe $H_L^\infty(A,\rho)$ as the set of elements $a\in A$ that lie in the domain of all $\{\partial^n\}_{n\in \N}$. Indeed, we note that the domain of $\partial^n$ is naturally given by
$$
{\rm dom}(\partial^n)=\{T\in  \mathbb B(L^2(A,\rho)): \partial(T),\ldots, \partial^{n-1}(T)\in {\rm dom}(\partial)\}.
$$
For $T\in{\rm dom}(\partial^n)$, we use the graph norm
\begin{equation*}
 q_n(T)=\sum_{r=0}^n\frac{1}{r!}\|\partial^r(T)\|.
\end{equation*}
Note that when equipped with the graph norm $q_n$, every subspace ${\rm dom}(\partial^n)$ is a Banach $^*$-algebra with isometric involution (see \cite{KiSh94}). We also note that 
\begin{equation}\label{continuity}
    q_n(\partial T)\leq (n+1)q_{n+1}(T), \quad\text{ holds for all }n\in\N.
\end{equation}
We are, however, ultimately interested in the Fr\'echet space given by the intersection of all the spaces $\{{\rm dom}(\partial^n)\}_{n\in\N}$ with $A$. This space will be denoted by
$$
H^{{\rm der},\infty}_L(A,\rho)=A\cap \bigcap_{n\in\N} {\rm dom}(\partial^n).
$$
Soon, we will show that $H^{{\rm der},\infty}_L(A,\rho)$ agrees with $H^{\infty}_L(A,\rho)$ whenever $A$ enjoys the rapid decay property. The next lemma will show that $\bigcup_{n\in\N} L_n$ is contained in $\bigcap_{n\in\N} {\rm dom}(\partial^n)$, which guarantees that $\bigcap_{n\in\N}{\rm dom}(\partial^n)$ is dense in $A$.

\begin{lem}\label{iterated}
Let $(A,\rho)$ be a tracial ${\rm C}^*$-probability space, and let
$L=(L_n)_{n=0}^\infty$ be a filtration of $A$.  If $a\in L_m$, then
\begin{equation*}
 p_jap_k=0
 \qquad\text{whenever}\qquad |j-k|>m.
\end{equation*}
Consequently, $\bigcup_{n\in\mathbb N}{L_n} \subset\bigcap_{k\in\mathbb N}{\rm dom}(\partial^k)$ and $ \|\partial^r(a)\|\leq m^r\|a\|$, when $a\in L_m$ and $r\geq1$. Furthermore, if $a\in A\cap{\rm dom}(\partial^r)$, then $a\in{\rm dom}((D_L-1)^r)$ and
\begin{equation*}
 \partial^r(a)\hat 1=i^r(D_L-1)^r\hat a=i^r\sum_{k\in\mathbb N}k^rp_k(\hat a).
\end{equation*}
\end{lem}

\begin{proof}
Let $a\in L_m$.  If $j>k+m$, then
$$
 a\bigl(p_kH\bigr)\subset a\overline{L_k}^{\,\|\cdot\|_2}\subset\overline{L_{m+k}}^{\,\|\cdot\|_2},
$$
and hence $p_jap_k=0$. On the other hand, if $k>j+m$, then $(p_jap_k)^*=p_ka^*p_j=0$, because $L_m$ is stable under $*$. This proves the first part of the lemma.

Now, for $-m\leq d\leq m$, set
\[
 a_d=\sum_{k\geq\max\{0,-d\}}p_{k+d}ap_k,
\]
where the sum converges in the strong operator topology. The previous observation implies that $a=\sum_{d=-m}^ma_d$ and
\begin{equation*}
 \beta_t(a)=\sum_{d=-m}^me^{idt}a_d,
\end{equation*}
which is a $\mathbb B(L^2(A,\rho))$-valued trigonometric polynomial of degree at most $m$. In particular, it is smooth as a function of $t$, which implies that $a\in \bigcap_{k\in\mathbb N}{\rm dom}(\partial^k) $. Furthermore, by the definition of the infinitesimal generator, we have that $\partial^r(a)=\frac{d^r}{dt^r}\beta_t(a)\big|_{t=0}$. Now, we apply the Bernstein inequality for trigonometric polynomials to see that
$$
\|\partial^r(a)\|\leq \sup_{t\in\R}\Big\|\frac{d^r}{dt^r}\beta_t(a) \Big\|\leq m^r \sup_{t\in\R}\|\beta_t(a)\|=m^r\|a\|.
$$

In order to prove the last assertion, we now suppose that $a\in A\cap{\rm dom}(\partial^r)$. Since $D_L(\hat 1)=\hat 1$, we have
\begin{equation*}
 \beta_t(a)\hat 1=U_taU_t^*\hat 1=e^{-it}U_t\hat a=e^{it(D_L-1)}\hat a.
\end{equation*}
Differentiating $r$ times, we get that
$$
 \partial^r(a)\hat1=\frac{d^r}{dt^r}\beta_t(a)\big|_{t=0}\hat 1=i^r(D_L-1)^r\hat a.
$$
Finally, $D_L-1$ acts as multiplication by $k$ on $p_kL^2(A,\rho)$, which gives the last identity.
\end{proof}

By using this lemma, we are now able to conclude that $H^{{\rm der},\infty}_L(A,\rho)$ and $H^{\infty}_L(A,\rho)$ define the same Fr\'echet space.

\begin{thm}\label{frechet-eq}
    Let $(A,\rho)$ be a tracial ${\rm C}^*$-probability space, and let $L=(L_{n})_{n=0}^{\infty}$ be a filtration of $A$ that has the rapid decay property. Then $H^{{\rm der},\infty}_L(A,\rho)$ coincides with $H^{\infty}_L(A,\rho)$ as Fr\'echet spaces.
\end{thm}
\begin{proof}
    Fix constants $\alpha>0$ and $C>1$ as in the definition of rapid decay. Set $c_0=\sum_{k\in\N}(1+k)^{-2}<+\infty$. Now, let $a\in H^{\infty}_L(A,\rho)$ and let $k\in \N$. Since $p_k(a)\in \overline{L_k}^{\,\|\cdot\|_2}$, by \cite[Corollary 3.7]{HKER25}, we actually have that $p_k(a)\in A$. Applying Lemma \ref{iterated} and a simple approximation argument, we get
    $$
    \|\partial^r(p_k(a))\|\leq C(1+k)^{r+\alpha}\|\widehat{p_k(a)}\|_2\qquad \text{for all } r\in\N.
    $$
    Consequently, 
    $$
    \sum_{k\in\mathbb N}\|\partial^r(p_k(a))\|\leq C\sum_{k\in\mathbb N}
     (k+1)^{r+\alpha}\|\widehat{p_k(a)}\|_2 \leq Cc_0\|\hat a\|_{H_L^{r+\alpha+1}}.
     $$ 
     This inequality implies that $\sum_{k\in \N}\partial^r(p_k(a)) $ converges in ${\rm dom}(\partial^r)$ and, appealing to linearity and closedness, we establish $\|\partial^r(a)\|\leq Cc_0\|\hat a\|_{H_L^{r+\alpha+1}}$. Finally, the first desired inequality follows from the fact that Sobolev norms increase with the index:
     \begin{align*}
         q_m(a)=\sum_{r=0}^m\frac{1}{r!}\|\partial^r(a)\|\leq c_0C\sum_{r=0}^m\frac{1}{r!}\|\hat a\|_{H_L^{r+\alpha+1}}\leq c_0Ce\|\hat a\|_{H_L^{m+\alpha+1}}.
     \end{align*}
On the other hand, if $a\in H^{{\rm der},\infty}_L(A,\rho)$, then 
    \begin{align*}
        \|\hat a-p_0(\hat a)\|_{H^{m}_L}&= \Big(\sum_{k=1}^\infty(1+k)^{2m}\|p_{k}(a)\|_2^2\Big)^{1/2} \\
        &\leq2^m\Big(\sum_{k\in\N}k^{2m}\|p_{k}(\hat a)\|_2^2\Big)^{1/2} \\
        &= 2^m\Big\|i^m\sum_{k\in\N}k^mp_{k}(\hat a)\Big\|_2 \\
        &=2^m\|\partial^m(a)\hat 1\|_2 \\
        &\leq 2^m\|\partial^m(a)\|.
    \end{align*}
    Hence $\|a\|_{H^{m}_L}\leq2^mm!q_m(a)$ and we conclude that $H^{\infty}_L(A,\rho)=H^{{\rm der},\infty}_L(A,\rho)$. The norm estimates also imply that $H^{{\rm der},\infty}_L(A,\rho)$ coincides with $H^{\infty}_L(A,\rho)$ as Fr\'echet spaces, as claimed.
\end{proof}

The theorem above, combined with a very classical result of Bratteli-Elliott-Jorgensen \cite[Lemma 3.2]{BaElJo84}, already proves that $H^\infty_L(A,\rho)$ is closed under the smooth functional calculus of $A$. In order to conclude, what remains is to show the norm bounds promised in the introduction. Our next lemma follows the work of Bratteli-Elliott-Jorgensen closely.

\begin{lem}[see {{\cite[Lemma 3.2]{BaElJo84}}}]\label{derivationsss}
Let $a=a^*$ belong to $\bigcap_{n\in\mathbb N}{\rm dom}(\partial^n)$. Then, for every $t\in\mathbb R$ we have
\begin{equation*}
 \partial(e^{ita})=i\int_0^t e^{isa}\partial(a)e^{i(t-s)a}\,ds.
\end{equation*}
More generally, for $r\geq1$,
\begin{align*}
 \partial^r(e^{ita})=i\sum_{j+k+l=r-1}\frac{(r-1)!}{j!k!l!}\int_0^t\partial^j(e^{isa})\partial^{k+1}(a)\partial^l(e^{i(t-s)a})\,ds.
\end{align*}
\end{lem}

\begin{proof}
Since $\partial$ is a derivation, we have 
$$
\partial(a^m)=\sum_{u=0}^{m-1}a^u\partial(a)a^{m-1-u}, \qquad\text{ for } m\geq1 \text{ and }a\in{\rm dom}(\partial).
$$
Now, because of \eqref{continuity}, the operator $\partial:{\rm dom}(\partial^k)\to{\rm dom}(\partial^{k-1})$ is continuous, and we can apply it term by term to the power series defining $e^{ita}$. It follows that
\begin{equation*}
 \partial(e^{ita})=\sum_{m=1}^\infty\frac{(it)^m}{m!}\sum_{u=0}^{m-1}a^u\partial(a)a^{m-1-u},
 \end{equation*}
where the series converges in ${\rm dom}(\partial^{r-1})$. Furthermore, using the following elementary identity
\begin{equation*}
 \int_0^t s^p(t-s)^q\,ds=\frac{p!q!}{(p+q+1)!}t^{p+q+1}, \quad \text{ for }p,q\geq0, t\in\R,
\end{equation*}
we derive the formula
\begin{align*}
 i\int_0^t e^{isa}\partial(a)e^{i(t-s)a}\,ds&=i\sum_{p,q\geq0}\frac{i^{p+q}}{p!q!}\left(\int_0^t s^p(t-s)^q\,ds\right)a^p\partial(a)a^q\\
 &=\sum_{m=1}^\infty\frac{(it)^m}{m!}\sum_{p+q=m-1}a^p\partial(a)a^q.
\end{align*}
Again, the formula is valid in every subalgebra of the form ${\rm dom}(\partial^{k-1})$.

Finally, for every $r\geq1$,  the map
\[
 s\longmapsto e^{isa}\partial(a)e^{i(t-s)a}
\]
is continuous with respect to $q_{r-1}$, so we may apply $\partial^{r-1}$ under the integral sign. Indeed, applying the Leibniz identity twice gives the following identity
\[
 \partial^{r-1}(xyz)
 =\sum_{j+k+l=r-1}\frac{(r-1)!}{j!k!l!}
   \partial^j(x)\partial^k(y)\partial^l(z).
\]
Taking $x=e^{isa}$, $y=\partial(a)$, and $z=e^{i(t-s)a}$ finishes our proof.
\end{proof}

The next lemma shows that every self-adjoint element in $\bigcap_{r\in\mathbb N}{\rm dom}(\partial^r)$ grows polynomially.

\begin{lem}\label{polgrowth}
For every $n\in\mathbb N$, there is a constant $E_n>0$, depending only on $n$, such that
\begin{equation*}
 q_n(e^{ita})
 \leq E_n(1+q_n(a))^n(1+|t|)^n
\end{equation*}
for every self-adjoint $a\in\bigcap_{r\in\mathbb N}{\rm dom}(\partial^r)$ and every $t\in\mathbb R$.
\end{lem}

\begin{proof}
The assertion is immediate for $n=0$.  Fix $n\geq1$ and set
\[
 M=1+\max_{1\leq r\leq n}\|\partial^r(a)\|.
\]
We claim that, for $0\leq r\leq n$, there are constants $c_r>0$ depending only on $r$ such that
\begin{equation}\label{induction}
 \|\partial^r(e^{ita})\|\leq c_rM^r(1+|t|)^r.
\end{equation}
For $r=0$, this follows from $\|e^{ita}\|=1$.  Suppose the assertion has been proved below order $r$. Lemma \ref{derivationsss} provides us with the formula
$$
\partial^r(e^{ita})=i\sum_{j+k+l=r-1}\frac{(r-1)!}{j!k!l!}\int_0^t\partial^j(e^{isa})\partial^{k+1}(a)\partial^l(e^{i(t-s)a})\,ds.
$$
If $s$ lies between $0$ and $t$, then $|s|,|t-s|\leq|t|$. The induction hypothesis therefore bounds the integrand by
\[
 c_jc_lM^{j+l+1}(1+|t|)^{j+l}.
\]
After integration, this is at most
\[
 c_jc_lM^r(1+|t|)^r,
\]
because $j+l+1=r-k\leq r$.  Summing the finitely many multinomial coefficients proves \eqref{induction}.

Finally,
\[
 M\leq1+n!q_n(a)\leq n!(1+q_n(a)),
\]
and summing \eqref{induction} over $0\leq r\leq n$ proves the required estimate.
\end{proof}

In order to state our main theorem, let us fix the following convention for the Fourier transform of a real function
\[
 \widehat f(t)=\int_{\mathbb R}f(s)e^{-its}\,ds.
\]

\begin{thm}\label{actual}
    Let $(A,\rho)$ be a tracial ${\rm C}^*$-probability space, and let
$L=(L_n)_{n=0}^\infty$ be a filtration of $A$ with the rapid decay property.
Then $H_L^\infty(A,\rho)$ is closed under the smooth functional calculus of
$A$.

More precisely, fix rapid decay constants $C,\alpha>0$.  For every
$n\in\mathbb N$, there is a constant $K_n>0$ such that, whenever
$a=a^*\in H_L^\infty(A,\rho)$ and $f\in C_c^\infty(\mathbb R)$,
\begin{equation*}
 f(a)=\frac{1}{2\pi}\int_{\mathbb R}\widehat f(t)e^{ita}\,dt
\end{equation*}
converges in every subspace $H_L^n(A,\rho)$, and we have the norm estimate
\begin{equation*}
 \|f(a)\|_{H_L^n}\leq K_n\bigl(1+\|a\|_{H_L^{\alpha+n+1}}\bigr)^n \int_{\mathbb R}|\widehat f(t)|(1+|t|)^n\,dt.
\end{equation*}
\end{thm}

\begin{proof}
Let $a=a^*\in H_L^\infty(A,\rho)$. By Theorem \ref{frechet-eq} and its proof, it is enough to show that $f(a)=\frac{1}{2\pi}\int_{\mathbb R}\widehat f(t)e^{ita}\,dt$ converges in each subspace of the form ${\rm dom}(\partial^n)$. But this follows from the polynomial growth estimate in Lemma \ref{polgrowth} and standard techniques \cite[Theorem 2.5]{Fl25a}.

The explicit estimate can be obtained directly: By Lemma \ref{polgrowth}, 
$$
q_n(f(a))\leq\frac{E_n}{2\pi}(1+q_n(a))^n\int_{\mathbb R}|\widehat f(t)|(1+|t|)^n\,dt.
$$
But the proof of Theorem \ref{frechet-eq} also shows that $\|f(a)\|_{H_L^n}\leq2^nn!q_n(f(a))$ and $q_n(a)\leq K\|a\|_{H_L^{n+\alpha+1}}$. The theorem is proved.
\end{proof}

\section*{AI statement}

GPT-5.6 Sol helped fix a mistake in an earlier version of Lemma \ref{iterated}. It also suggested a preliminary form of Lemma \ref{derivationsss}, which was ultimately traced back to the work of Bratteli-Elliott-Jorgensen by the author. The rest of the mathematical content, including all of the ideas behind the conception of this article, is due to the author. Every word in this document has been written by its author.

\section*{Acknowledgments}

The author gratefully acknowledges support from the Simons Foundation Dissertation Fellowship SFI-MPS-SDF-00015100. He also thanks Ben Hayes for initial discussions, many comments, and his guidance.

\printbibliography

\bigskip
\bigskip
ADDRESS

\smallskip
\smallskip
Felipe Flores

Department of Mathematics, University of Virginia,

114 Kerchof Hall. 141 Cabell Dr,

Charlottesville, Virginia, United States

E-mail: hmy3tf@virginia.edu

\end{document}